\documentclass[12pt,draft,reqno]{amsart}
\makeatletter
\providecommand*{\shuffle}{%
  \mathbin{\mathpalette\shuffle@{}}%
}
\newcommand*{\shuffle@}[2]{%
  \sbox0{$#1\vcenter{}$}%
  \kern .15\ht0 
  \rlap{\vrule height .25\ht0 depth 0pt width 2.5\ht0}%
  \raise.1\ht0\hbox to 2.5\ht0{%
    \vrule height 1.75\ht0 depth -.1\ht0 width .17\ht0 %
    \hfill
    \vrule height 1.75\ht0 depth -.1\ht0 width .17\ht0 %
    \hfill
    \vrule height 1.75\ht0 depth -.1\ht0 width .17\ht0 %
  }%
  \kern .15\ht0 
}
\makeatother

\usepackage{hyperref}
\usepackage{amsmath}
\usepackage{amssymb}
\usepackage{amsfonts}
\usepackage{amsthm}
\usepackage{ascmac}
\usepackage{bm}
\usepackage{url}
\numberwithin{equation}{section}
\newtheorem{theorem}{Theorem}[section]
\newtheorem{lemma}[theorem]{Lemma}

\newtheorem{example}[theorem]{Example}

\theoremstyle{plain}
\newtheorem{theo}{Theorem}

\newtheorem{prop}[theo]{Proposition}

\theoremstyle{definition}

\newcommand{\RR}{\mathfrak{R}}
\newcommand{\F}{\mathcal{F}}
\newcommand{\N}{\mathbb{N}}
\newcommand{\Z}{\mathbb{Z}}
\newcommand{\Q}{\mathbb{Q}}

\newcommand{\C}{\mathbb{C}}

\allowdisplaybreaks
\address[Daiki Iju]{Department of Mathematical Science, Graduate School of Science and Technology, Tokyo University of Science 2641 Yamazaki, Noda-shi, Chiba Prefecture 278-8510, Japan}
\email{6126701@ed.tus.ac.jp}
\address[Takashi Nakamura]{Department of Mathematical Science, Graduate School of Science and Technology, Tokyo University of Science 2641 Yamazaki, Noda-shi, Chiba Prefecture 278-8510, Japan}
\email{nakamuratakashi@rs.tus.ac.jp}

\begin{document}
\title[Double $L$-values of Negative Weight Admitting Series Representations]{Double \textit{L}-values of Negative Weight Admitting Series Representations}
\author[D.~IJU and T.~NAKAMURA]{Daiki Iju and Takashi NAKAMURA }
\maketitle

\begin{abstract}
In this paper, we construct double $L$-values of negative weight that admit
series representations.
In our construction, we use the Thue–Morse sequence, which appears in combinatorics on words and automatic sequences.
\end{abstract}


\section{ Introduction and Main result}
In 2004,  Arakawa and Kaneko \cite[Section 1]{AK} introduced the following two types of multiple $L$-values (see also \cite[Chapter 13]{ZhaoPL}).
Fix $m\in\N$ and  $\mathfrak{R} = \mathfrak{R}_m$ denote the $\Z$-module $\mathbb{Z}/m\mathbb{Z}$. 
Let $\mathcal{F}(\mathfrak{R};\mathbb{C})$ be the $\C$-vector space consisting of all mappings $f:\RR\to\C$. 
An element $f\in \F(\mathfrak{R};\C)$ is viewed naturally as a function on $\Z$ via the projection $\Z\to\RR$.
Fix a primitive $m$-th root of unity $\xi = \xi_m:=\exp(2\pi i/m)$. Then, for each  $a\in \mathfrak{R}$, let  $\varphi_a \in \mathcal{F}(\mathfrak{R};\mathbb{C})$, we defined by
\[
\varphi_a(x)=\xi^{ax}\quad (x\in \mathfrak{R}).
\]
Note that the set of functions $\{ \varphi_a \}_{a\in \mathfrak{R}}$ constitutes a basis of the space $\F(\mathfrak{R};\C)$. Moreover, each $f\in\F(\RR;\C)$ admits a finite Fourier expansion 
\begin{equation}\label{f1}
f(x)=\sum_{a\in\RR}\hat{f}(a)\xi^{ax}, \quad \hat{f}(a)=\frac{1}{m}\sum_{y\in\RR}f(y)\xi^{-ax}.
\end{equation}
For $f_1,\ldots,f_n \in \mathcal{F}(\mathfrak{R};\mathbb{C})$, $k_1,\ldots,k_n\in\N$ and $k_1\ge 2$, we define the following multiple $L$-values
$L_{\shuffle}$ and $L_*$ by
\[
\begin{aligned}
L_{\shuffle}(k_1,\ldots,k_n;&f_1,\ldots,f_n) \\
&:=\sum_{\substack{m_1>\cdots>m_n>0}}\frac{f_1(m_1-m_2)\cdots f_{n-1}(m_{n-1}-m_n)f_n(m_n)}{m_1^{k_1}m_2^{k_2}\cdots m_n^{k_n}}
\end{aligned}
\]
and
\[
L_\ast(k_1,\ldots,k_n;f_1,\ldots,f_n):=\sum_{\substack{m_1>\cdots>m_n>0}}\frac{f_1(m_1)f_2(m_2)\cdots f_n(m_n)}{m_1^{k_1}m_2^{k_2}\cdots m_n^{k_n}}.
\]
We call $k=k_1+\cdots+k_n$ the weight and $n$ the depth. If $k_1\ge 2$, these series converge absolutely. 
When $k_1 = 1$, the series are presented to be the limits
\[
\begin{aligned}
L_{\shuffle}(&1,k_2,\ldots,k_n;f_1,\ldots,f_n) \\
&:=\lim_{R\to \infty}\sum_{\substack{R\ge m_1>\cdots>m_n>0}}\frac{f_1(m_1-m_2)\cdots f_{n-1}(m_{n-1}-m_n)f_n(m_n)}{m_1m_2^{k_2}\cdots m_n^{k_n}},
\end{aligned}
\]
\[
\begin{aligned}
L_\ast(&1,k_2,\ldots,k_n;f_1,\ldots,f_n) \\
&:=\lim_{R\to \infty}\sum_{\substack{R\ge m_1>\cdots>m_n>0}}\frac{f_1(m_1)f_2(m_2)\cdots f_n(m_n)}{m_1m_2^{k_2}\cdots m_n^{k_n}}.
\end{aligned}
\]
As for the convergence, there is the following criterion in \cite[Proposition 1.1]{AK}.
\begin{prop}[{\cite[Proposition 1.1]{AK}}]\label{prop1}
Suppose $k_1=1$. Then, the series $L_{\shuffle}$ and $L_\ast$ are convergent if and only if $\sum_{y\in \mathfrak{R}}f_1(y)=0$.
\end{prop}
However, Ageji, Iju, and Nakamura \cite{AIN} provided a counter  example  to this claim. More precisely, they showed the following result. 

\begin{theo}[{\cite[Theorem 1.1]{AIN}}]\label{th}
Let $j\le 1$ be an integer and, $f_1 = 1$ which implies $\sum_{y\in \mathfrak{R}_2}f_1(y)\neq 0$. 
Then, there exist $k\in \mathbb{N}_{\ge 2}$ and $f_2 \in \mathcal{F}(\mathfrak{R}_2;\mathbb{Q})$ 
such that $j+k\ge3$ and the series $L_\ast(j,k;f_1,f_2)$ converges.
\end{theo}
Motivated by this result, we obtain the following theorem.
This result shows that one can construct convergent multiple $L$-values even in the case of negative weight.
\begin{theorem}\label{th1}
For any $l\in\N_{\ge0}$, there exist $f_1,f_2\in\F(\RR_{2^{l+3}};\Q(\xi_{2^{l+2}}))$ such that the series $L_\ast(-l,2;f_1,f_2)$ converges.
\end{theorem}
We remark that the sequences $f_1$ and $f_2$ constructed in Theorem \ref{th1} take non-zero values at infinitely many integers (see Lemma \ref{lem1}).

\section{Preliminaries}
In this section, we provide some preliminary results used in the proof of Theorem \ref{th1}.
Let us define the sets $O$ and $E$ by
\[
O := \{n \in\N_{\ge0} \mid n\  \text{has an odd number of 1's in its binary expansion}\},
\]
\[
E := \{n \in\N_{\ge0} \mid n\  \text{has an even number of 1's in its binary expansion}\}.
\]
We assume that $0\in E$.
Thue-Morse sequence $t_j(j\in\N_{\ge0})$ is defined by
\begin{equation}\label{t1}
t_j:=
\begin{cases}
1 & j\in O, \\
0 & j\in E.
\end{cases}
\end{equation}
In number theory, there is a partition problem known as the Prouhet--Tarry--Escott problem.
It concerns subsets $A,B\subset\Z$ that satisfy a system of $n$ simultaneous Diophantine system:
\[
\sum_{a\in A}a^k=\sum_{b\in B}b^k\quad(k=1,2,\ldots,n).
\]
In 1851, Prouhet found the following solution.
\begin{theo}[e.g.~{\cite[Theorem 6]{P}}]\label{tp}
The Thue-Morse sequence $t_j$ has the following property. For any $l\in\N_{\ge0}$, we define the sets $A$ and $B$ by
\[
A:=\{ a\in\{0,1,2,\ldots,2^l-1 \}:t_a=1\},
\]
\[
B:=\{ b\in\{0,1,2,\ldots,2^l-1 \}:t_b=0\}.
\]
Then, for all $0\le k\le l-1$, we have
\begin{equation}\label{tp0}
\sum_{a\in A}a^k=\sum_{b\in B}b^k,
\end{equation}
with the convention that  $0^0=1$. In other words, 
\begin{equation}\label{45}
\sum_{j=0}^{2^l-1}(-1)^{t_j}j^k=0.
\end{equation}
\end{theo}
\begin{example}
In (\ref{tp0}), when $l=2$,
\begin{align*}
1^0+2^0&=0^0+3^0=2, \\
1^1+2^1&=0^1+3^1=3.
\end{align*}
When $l=3$,
\begin{align*}
1^0+2^0+4^0+7^0&=0^0+3^0+5^0+6^0=4, \\
1^1+2^1+4^1+7^1&=0^1+3^1+5^1+6^1=14, \\
1^2+2^2+4^2+7^2&=0^2+3^2+5^2+6^2=70.
\end{align*}
\end{example}

\section{Proof}
For any $l\in\N_{\ge0}$, we define the sequence $f_1\in\F(\RR_{2^{l+3}};\Q)$ by
\[
f_1(m):=\sum_{k=0}^{2^{l+3}-1}\xi_{2^{l+3}}^{km}=
\begin{cases}
1 & m\equiv 0\pmod{2^{l+3}}, \\
0 & \text{otherwise}.
\end{cases}
\]
Then $\sum_{y\in \RR_{2^{l+3}}}f_1(y)\neq0$, and $f_1$ has a Fourier expansion of the form (\ref{f1}). 
For any $0\le j\le2^l-1$, we define the sequences $g_{j1},g_{j2}\in\F(\RR_{2^{l+3}};\Q(\xi_{2^{l+2}}))$ and $h_j\in\F(\RR_{2^{l+2}};\Q(\xi_{2^{l+2}}))$ by
\begin{equation}\label{41}
g_{j1}(m):=\frac{1}{2^{l+1}}\sum_{k=0}^{2^{l+1}-1}\xi_{2^{l+2}}^{-j(4k+1)}\xi_{2^{l+3}}^{(4k+1)m} ,
\end{equation}
\begin{equation}\label{42}
g_{j2}(m):=\frac{1}{2^{l+1}}\sum_{k=0}^{2^{l+1}-1}\xi_{2^{l+2}}^{-j(4k+3)}\xi_{2^{l+3}}^{(4k+3)m}, 
\end{equation}
\begin{equation}\label{43}
h_j(m):=\frac{1}{2^{l+2}}\sum_{k=0}^{2^{l+1}-1}\xi_{2^{l+2}}^{-j(2k+1)}\xi_{2^{l+2}}^{(2k+1)m}. 
\end{equation}
Using $g_{j1},g_{j2},h_j$ and (\ref{t1}), define the sequence $f_2\in\F(\RR_{2^{l+3}};\Q(\xi_{2^{l+2}}))$ by
\[
f_2:=\sum_{j=0}^{2^l-1}(-1)^{t_{j}}(g_{j1}+g_{j2}-h_j).
\]
The sequence $f_2$ has the following basic properties.
\begin{lemma}\label{lem1}
The sequence $f_2$ is not identically zero. More precisely, the cardinality of set 
\[
S(f_2)=\{m\in\N_{\ge0}\mid f_2(m)\neq0 \}
\]
is infinite.
\end{lemma}
\begin{proof}
By the definition of $g_{j1}$, we have
\[
\begin{aligned}
g_{j1}(m)&=\frac{\xi_{2^{l+3}}^{m-2j}}{2^{l+1}}\sum_{k=0}^{2^{l+1}-1}\xi_{2^{l+3}}^{4k(m-2j)}=\frac{\xi_{2^{l+3}}^{m-2j}}{2^{l+1}}\sum_{k=0}^{2^{l+1}-1}\xi_{2^{l+1}}^{k(m-2j)} \\
&=
\begin{cases}
i^{(m-2j)/2^{l+1}} & m\equiv 2j\pmod{2^{l+1}}, \\
0 & \text{otherwise}.
\end{cases}
\end{aligned}
\]
By a similar calculation, we obtain
\[
\begin{aligned}
g_{j2}(m)&=\frac{\xi_{2^{l+3}}^{3(m-2j)}}{2^{l+1}}\sum_{k=0}^{2^{l+1}-1}\xi_{2^{l+3}}^{4k(m-2j)}=\frac{\xi_{2^{l+3}}^{3(m-2j)}}{2^{l+1}}\sum_{k=0}^{2^{l+1}-1}\xi_{2^{l+1}}^{k(m-2j)} \\
&=
\begin{cases}
(-i)^{(m-2j)/2^{l+1}} & m\equiv 2j\pmod{2^{l+1}}, \\
0 & \text{otherwise},
\end{cases}
\end{aligned}
\]
\[
\begin{aligned}
h_j(m)&=\frac{\xi_{2^{l+2}}^{m-j}}{2^{l+2}}\sum_{k=0}^{2^{l+1}-1}\xi_{2^{l+2}}^{2k(m-j)}=\frac{\xi_{2^{l+2}}^{m-j}}{2^{l+2}}\sum_{k=0}^{2^{l+1}-1}\xi_{2^{l+1}}^{k(m-j)} \\
&=
\begin{cases}
(-1)^{(m-j)/2^{l+1}}/2 & m\equiv j\pmod{2^{l+1}}, \\
0 & \text{otherwise}.
\end{cases}
\end{aligned}
\]
Thus, we have 
\[
f_2(1)=(-1)^{t_1}\left(-h_1(1)\right)=1/2.
\]
Since $f_2$ is a periodic sequence with period $2^{l+3}$, it follows $f_2(1)=f_2(1+2^{l+3}T)=1/2$ for any $T\in\N$.
\end{proof}
\begin{theorem}\label{th3}
For the sequences $f_1,f_2\in\F(\RR_{2^{l+3}};\Q(\xi_{2^{l+2}}))$ defined above, the infinite series $L_\ast(-l,2;f_1,f_2)$ converges.
\end{theorem}
\begin{proof}
For any $0\le j\le2^l-1$, we consider the following $2^l$ equations:
\begin{equation}\label{1}
\begin{aligned}
\sum_{m=1}^{\infty}\frac{i^m}{\left(2^{l+1}m+2j\right)^2}+\sum_{m=1}^{\infty}\frac{(-i)^m}{\left(2^{l+1}m+2j\right)^2} 
=\sum_{m=1}^{\infty}\frac{2(-1)^m}{\left(2^{l+2}m+2j\right)^2} 
=\sum_{m=1}^{\infty}\frac{(-1)^m/2}{\left(2^{l+1}m+j\right)^2}.
\end{aligned}
\end{equation}
For any $0\le j\le2^l-1$, by the definitions of $g_{j1}$, $g_{j2}$ and $h_j$, the equation (\ref{1}) can be rewritten as 
\[
\sum_{m=1}^{\infty}\frac{g_{j1}(m)}{m^2}+\sum_{m=1}^{\infty}\frac{g_{j2}(m)}{m^2}=\sum_{m=1}^{\infty}\frac{h_j(m)}{m^2}.
\]
Clearly, the equation above implies
\begin{equation}\label{44}
L_\ast(2;g_{j1}+g_{j2}-h_j)=\sum_{m=1}^{\infty}\frac{g_{j1}(m)+g_{j2}(m)-h_j(m)}{m^2}=0.
\end{equation}
From (\ref{44}), we have
\[
L_\ast(2;f_2)=\sum_{j=0}^{2^{l}-1}(-1)^{t_j}L_\ast(2;g_{j1}+g_{j2}-h_j)=0.
\]
Hence, for any $N\ge2$, we obtain
\begin{equation}\label{2}
\sum_{m=1}^{2^{l+3}N-1}\frac{f_2(m)}{m^2}=-\sum_{m=2^{l+3}N}^{\infty}\frac{f_2(m)}{m^2}.
\end{equation}
Now, we consider  the estimation of the right-hand side of (\ref{2}). Here we decompose the right-hand side of (\ref{2}) as follows:
\begin{equation}\label{46}
\sum_{m=2^{l+3}N}^{\infty}\frac{f_2(m)}{m^2}=H(N)+G_1(N)+G_2(N),
\end{equation}
where
\[
H(N):=-\sum_{m=2^{l+3}N}^{\infty}\left(\sum_{j=0}^{2^l-1}(-1)^{t_j}h_j(m)\right)\frac{1}{m^2}, 
\]
\[
G_1(N):=\sum_{m=2^{l+3}N}^{\infty}\left( \sum_{j=0}^{2^l-1}(-1)^{t_{j}}g_{j1}(m)\right)\frac{1}{m^2}, 
\]
\[
G_2(N):=\sum_{m=2^{l+3}N}^{\infty}\left( \sum_{j=0}^{2^l-1}(-1)^{t_{j}}g_{j2}(m)\right)\frac{1}{m^2}. 
\]
We first estimate $H(N)$. From the right-hand side of (\ref{1}), we have 
\[
H(N)=-\sum_{m=4N}^{\infty}\left(\frac{(-1)^m}{2}\sum_{j=0}^{2^l-1}\frac{(-1)^{t_{j}}}{\left(2^{l+1}m+j\right)^2}\right).
\]
The summand above is rewritten as follows:
\begin{equation}\label{h}
\sum_{j=0}^{2^l-1}\frac{(-1)^{t_{j}}}{\left(2^{l+1}m+j\right)^2}=\frac{1}{(2^{l+1}m)^2}\sum_{j=0}^{2^l-1}(-1)^{t_{j}}\left(1+\frac{j}{2^{l+1}m}\right)^{-2}.
\end{equation} 
From the binomial theorem, we have
\begin{equation}\label{bi1}
\begin{aligned}
(-1)^{t_{j}}\left(1+\frac{j}{2^{l+1}m}\right)^{-2}&=(-1)^{t_{j}}\left\{\sum_{k=0}^{l-1}\binom{-2}{k}\left(\frac{j}{2^{l+1}m}\right)^k\right\} \\
&\quad+(-1)^{t_j}\binom{-2}{l}\left(\frac{j}{2^{l+1}m}\right)^{l}+O(m^{-(l+1)}).
\end{aligned}
\end{equation}
From Theorem \ref{tp}, we obtain
\[
\begin{aligned}
\sum_{j=0}^{2^l-1}(-1)^{t_{j}}\left\{\sum_{k=0}^{l-1}\binom{-2}{k}\left(\frac{j}{2^{l+1}m}\right)^k\right\} 
=\sum_{k=0}^{l-1}\binom{-2}{k}\frac{1}{2^{k(l+1)}m^k}\left(\sum_{j=0}^{2^l-1}(-1)^{t_{j}}j^k\right)=0.
\end{aligned}
\]
Thus, the equation (\ref{h}) can be expressed as
\[
\begin{aligned}
\sum_{j=0}^{2^l-1}\frac{(-1)^{t_{j}}}{\left(2^{l+1}m+j\right)^2}&=\frac{1}{(2^{l+1}m)^2}\left(\sum_{j=0}^{2^l-1}(-1)^{t_j}\binom{-2}{l}\left(\frac{j}{2^{l+1}m}\right)^{l} +O(m^{-(l+1)})\right) \\
&=\frac{X_l}{m^{l+2}}+O(m^{-(l+3)}),
\end{aligned}
\]
where
\[ 
X_l:=\sum_{j=0}^{2^l-1}(-1)^{t_j}\binom{-2}{l}\frac{j^l}{2^{(l+1)(l+2)}}.
\]
Therefore, $H(N)$ can be rewritten as  
\begin{equation}\label{w}
\begin{aligned}
H(N)&=\sum_{m=4N}^{\infty}\left(\frac{-X_l}{2}\frac{(-1)^m}{m^{l+2}}+O(m^{-(l+3)})\right) \\
&=\frac{-X_l}{2}\left(\sum_{m=4N}^{\infty}\frac{(-1)^m}{m^{l+2}}\right)+O(N^{-(l+2)}).
\end{aligned} 
\end{equation}
Moreover, the sum on the right-hand side of (\ref{w}) becomes
\begin{equation}\label{x}
\begin{aligned}
\sum_{m=4N}^{\infty}\frac{(-1)^m}{m^{l+2}}&=\sum_{m=N}^{\infty}\left(\frac{1}{(2m)^{l+2}}-\frac{1}{(2m+1)^{l+2}}\right)\\
&=\sum_{m=2N}^{\infty}\frac{1}{(2m)^{l+2}}\left(1-\left(1+\frac{1}{2m}\right)^{-(l+2)} \right).
\end{aligned}
\end{equation}
Furthermore, by the binomial theorem, one has $\left(1+\frac{1}{2m}\right)^{-(l+2)}=1+O(m^{-1})$. Thus, 
by (\ref{w}) and (\ref{x}), we obtain the estimation
\begin{equation}\label{y}
H(N)=\sum_{m=2N}^{\infty}O(m^{-(l+3)})\ll \int_{2N}^{\infty}\frac{1}{x^{l+3}}dx\ll N^{-(l+2)}.
\end{equation}
Next, we estimate $G_1(N)$. By the same method as for $H(N)$ and
\[
\begin{aligned}
(-1)^{t_{j}}\left(1+\frac{j}{2^{l}m}\right)^{-2}&=(-1)^{t_{j}}\left\{\sum_{k=0}^{l-1}\binom{-2}{k}\left(\frac{j}{2^{l}m}\right)^k\right\} \\
&\quad+(-1)^{t_j}\binom{-2}{l}\left(\frac{j}{2^{l}m}\right)^{l}+O(m^{-(l+1)})
\end{aligned}
\]
instead of (\ref{bi1}), the sum $G_1(N)$ can be expressed as  
\[
G_1(N)=\sum_{m=4N}^{\infty}\left(Y_l\frac{i^m}{m^{l+2}}+O(m^{-(l+3)})\right)=Y_l\left( \sum_{m=4N}^{\infty}\frac{i^m}{m^{l+2}}\right)+O(N^{-(l+2)}),
 \]
 where
\[ 
Y_l:=\sum_{j=0}^{2^l-1}(-1)^{t_j}\binom{-2}{l}\frac{j^l}{2^{l^2+2l+2}}.
\]
In addition, by an argument similar to that in (\ref{x}), we have
\[
\begin{aligned}
\left\lvert\sum_{m=4N}^{\infty}\frac{i^m}{m^{l+2}}\right\rvert&=\left\lvert\sum_{m=2N}^{\infty}\left(\frac{i^{2m}}{(2m)^{l+2}}+\frac{i^{2m+1}}{(2m+1)^{l+2}}\right)\right\rvert\\
&=\left\lvert\sum_{m=2N}^{\infty}\frac{(-1)^m}{(2m)^{l+2}}+i\sum_{m=2N}^{\infty}\frac{(-1)^m}{(2m+1)^{l+2}}\right\rvert \\
&\le \left\lvert\sum_{m=2N}^{\infty}\frac{(-1)^m}{(2m)^{l+2}}\right|+\left|\sum_{m=2N}^{\infty}\frac{(-1)^m}{(2m+1)^{l+2}}\right\rvert \ll N^{-(l+2)}.
\end{aligned}
\]
Hence, we obtain the estimate
\[
G_1(N)\ll N^{-(l+2)}.
\]
Since the estimation for $G_2(N)$ is obtained simply by replacing $i$ with $-i$ in the estimation for $G_1(N)$, one has
\[
G_2(N)\ll N^{-(l+2)}. 
\]
Hence, we have 
\begin{equation}\label{48}
\sum_{m=2^{l+3}N}^{\infty}\frac{f_2(m)}{m^2}=H(N)+G_1(N)+G_2(N)\ll N^{-(l+2)}.
\end{equation}
From the definition of  $L_\ast(-l,2;f_1,f_2)$, it holds that
\[
L_\ast(-l,2;f_1,f_2)=\sum_{n>m>0}\frac{f_1(n)f_2(m)}{n^{-l}m^2}=\sum_{n=1}^{\infty}\frac{1}{(2^{l+3}n)^{-l}}\sum_{m=1}^{2^{l+3}n-1}\frac{f_2(m)}{m^2}.
\]
By using (\ref{2}), we have
\[
L_\ast(-l,2;f_1,f_2)=\sum_{n=1}^{\infty}\frac{1}{(2^{l+3}n)^{-l}}\left(-\sum_{m=2^{l+3}n}^{\infty}\frac{f_2(m)}{m^2}\right).
\]
Thus, from the estimation (\ref{48}), we obtain
\[
\lvert L_\ast(-l,2;f_1,f_2)\rvert\ll\sum_{n=1}^{\infty}\frac{1}{n^{-l}}n^{-l-2}= \sum_{n=1}^{\infty}\frac{1}{n^2}<\infty .
\]
Therefore, we construct double $L$-values that converge even in the case of negative weight.
\end{proof}
\section*{Acknowledgments}
The second author was partially supported by JSPS grant 22K03276. We would like to express our sincere gratitude to Mr.~Hayashida for his valuable advice and helpful comments.

 
\end{document}